\documentclass{amsart}
\usepackage{graphicx} 
\usepackage{amssymb,amsthm,amsmath,mathtools,hyperref,float,comment,soul,xcolor,verbatim,graphicx,booktabs}
\usepackage{algorithm,algorithmicx,algcompatible,multirow,booktabs,array,float}
\usepackage{hyperref}
\usepackage{cleveref}
\usepackage[left=3cm, right=3cm, top=2.5cm, bottom=2.5cm]{geometry}
\newtheorem{theorem}{Theorem}[section]
\newtheorem{proposition}[theorem]{Proposition}

\newtheorem{corollary}[theorem]{Corollary}
\newtheorem{definition}[theorem]{Definition}
\newtheorem{remark}[theorem]{Remark}

\title{Regularization and compression of integral operator kernels with tensor paraproducts}
\author{Oluwadamilola Fasina}
\address{Department of Applied and Computational Mathematics \\ Yale University \\ New Haven, CT 06511 \\ USA}
\email{dami.fasina@yale.edu}

\author{Ronald R. Coifman}
\address{Department of Mathematics \\ Yale University \\ New Haven, CT 06511 \\ USA}
\email{coifman-ronald@yale.edu}
\begin{document}

\maketitle

\begin{abstract}

We present a new perspective on integral operator kernels given by a nonlinear function of the distance. This perspective is useful since it permits one to quasilinearize the nonlinear transformations acting on the tensor Haar expansion of the distance, d(x,y). This leads to a new representation comprised of a principal and residual term with desirable features; namely, enhanced local regularity of the principal component (and in certain situations rapid off-diagonal decay leading to O(N) storage) and a residual expansion of improved approximation as a consequence of enhanced global regularity. Numerical experiments on the potential kernel are included \footnotemark
\end{abstract}

\footnotetext{\href{https://github.com/obfasina/Regularization_compression_IOP_kernels/tree/main/Compression_regularization_TPA}{https://github.com/obfasina}}

\section{Introduction}

A common situation in computational mathematics \cite{nelson2026bridging,psenka2026parallel,demanet2007wave,madhu2026heist,lindsey2026goforth,fesserunitary,mateo2026lu,wilber2025time,trefethen2025numerical,dent2026controlling,cai2026globally,kondor2025principles,gopal2025highly,zhang2024finding} is the evaluation of integral operator kernels against functions. The potential, $K_n(x,y)$, and fractional Cauchy, $K_{\alpha}(r,\theta)$, kernels

\begin{align}
K_n(x,y) \coloneq log(\sum_{k=0}^n a_k d(x,y)^{-k} )  \qquad K_{\alpha}(r,\theta) \coloneq (re^{i \theta} - z)^{-\alpha} 
\label{eq:1}
\end{align}

arise in various physical and mathematical scenarios \cite{greengard2021fast,vico2016decoupled,barnett2011new,gimbutas2003generalized,yarvin1998generalized,cui2025enhanced,hibschweiler2020fractional}. We outline how these kernels (associated with the integral operators $T_\alpha: f(\theta) \to  \int K_{\alpha}(r,\theta) f(\theta)$ and $T_n: f(\theta) \to  \int K_{n}(x,y) f(y)$) can be regularized and compressed by means of a paradigm shift contingent on priors from harmonic analysis \cite{schlag2007remark,goldstein2001holder,erdougan2008strichartz,coifman1985some,coifman1986nonlinear,coifman2011harmonic,ostermann2022fourier,marsden2026splitting,gilbert2003one} and more specifically, paradifferential calculus \cite{muscalu2004bi,muscalu2006multi,bony1981calcul,fasina2025quasilinearization,daubechies1997harmonic,alazard2009paralinearization,alazard2024paracomposition}, enabling rapid evaluation of the range of $T_\alpha,T_n$. In particular, understanding integral operators kernels as manifestations of affinities, $d(x,y)$, composed with smooth functions permits one to reason about the frequency content of the affinity and the function it is composed with separately so long both functions have reasonable regularity. \\

$K_n(x,y)$, prima facie, has slow off-diagonal decay when $d(x,y) < 1$, resulting in a dense matrix with $O(N^2)$ storage; however, splitting $K_n(x,y)$ into two separate functions $\varphi_1 : f \to log(f)$ and $d(x,y) : [0,1]^1 \to [0,1]$ leads to a principal component in the decomposition with $O(N)$ storage, as rapid off-diagonal decay is realized through composing the low frequency content of $\texttt{supp}(\varphi_1)$ with the high frequency content of $\varphi_1$, leading to an integral operator compression technique with a perfume of the FWT method \cite{beylkin1991fast} and the like \cite{beylkin1991fast,grengard2006rapid,hackbusch1999sparse,beylkin1992representation,alpert1993wavelet,kaye2018transparent,wang2019monarch}. \\

Specifically, the insight that certain integral operator kernels can be split into two functions, can be exploited with tensor paraproducts - a program initiated in \cite{fasina2025quasilinearization} and further developed in \cite{fasina2026hierarchical,fasina2025d} - which allows one to differentially reason about the frequency content of function compositions with respect to their support. This leads to a principal term comprised of wavelet coefficients of the inner function composed with derivatives of an outer function multiplied by wavelet coefficients of the inner function, resembling the non-standard form described in the groundbreaking FWT paper \cite{beylkin1991fast}, and a residual term with improved global regularity by virtue of the tensor paraproduct decomposition. The tensor paraproduct decomposition can be applied to the kernels, $K_n(x,y)$ and $K_\alpha(r,\theta)$ in two ways: \emph{direct} and \emph{indirect} quasilinearization. \\

The \emph{direct} quasilinearization is discussed abstractly then specifically with the kernel $K_n(x,y)$. Its consequences are (i) the principal term, which contains high-low paraproducts (e.g. high freuqency content of the outer function acting on the low freuqency content of the inner function) and tensor wavelet expansions of the affinity, has improved local regularity in the sense of the Besov norm (ii) the principal term can compressed directly (as its representation is comprised of wavelet coefficients) when  the kernel has rapid off-diagonal decay by thresholding its matrix entries (iii) the residual term has increased global regularity, which comes for free as a consequence of the tensor paraproduct decomposition. For the \emph{indirect} quasilinearization, we describe how applying a smooth scalar function to $K_{\alpha}(r,\theta)$ leads to regularity gains due to the regularizing property of the tensor paraproduct decompostion, which again comes for free by virtue of \cite{fasina2025quasilinearization}. \\

The paper is structured as follows: in section 4 the relevant expository devices are introduced, and a corollary regarding the enhanced local regularity of kernels is stated and proved. Section 5 formalizes the direct and indirect quasilinearization methods, section 6 discussed the algorithms and complexity of the direct quasilinearization method, section 7 elaborates on the methods described in section 5 for the specific kernels, $K_n(x,y) \coloneq log(\sum_{k=0}^n a_k d(x,y)^{-k} )$ and $K_{\alpha}(r,\theta) \coloneq (re^{i \theta} - z)^{-\alpha} $, and section 8 concludes with numerical experiments.

\section{Acknowledgements}

The authors would like to thank Vladimir Rokhlin for helpful suggestions.

\section{Declaration of AI usage}

The authors are responsible for all ideation, mathematical reasoning, manuscript preparation, and most numerical computation. Claude was used only to assist with minor utility functions in Python, for which the authors also claim responsibility for careful verification.

\section{Tensor Paraproducts}

First we recall how tensor paraproducts \cite{fasina2025quasilinearization} can be used to approximate compositions of the form $A(f)$ for the situation $(A \in C^2(\mathbb{R}), f \in \Lambda_{\alpha}([0,1]^2))$ which is a natural setting for matrices, then we discuss a corollary to the main result of \cite{fasina2025quasilinearization} which has relevance to our proposed numerical method  \\

J.M. Bony's seminal work \cite{bony1981calcul} was expanded in  \cite{fasina2025quasilinearization,fasina2025d,fasina2026hierarchical} to tensor paraproducts, which permits finer analytical control for function compositions (of a certain regularity) by allocating distinct scaling parameters for each axis of the composition support. While \cite{fasina2025d} generalizes this technology to higher dimensions and \cite{fasina2026hierarchical} for non-Euclidean support, \cite{fasina2025quasilinearization} is useful for the matrix setting; we enunciate its main result below:

\begin{theorem}
Suppose  $A \in \mathcal{C}^2(\mathbb{R})$, $ f \in  \Lambda_\alpha([0,1]^2), 0 < \alpha < \frac{1}{2}$, then for the operator $T: f \to A(f)$ we can approximate $A(f)$ with

\begin{align}
&  \tilde{A}_{(N,N')}(f) =  \sum_{j=0}^{N} \sum_{j'=0}^{N'}  A'(P^jP'^{j'}(f)) Q^jQ'^{j'}(f) + A''(P^jP'^{j'}(f))Q^jP'^{j'}(f)P^jQ'^{j'}(f)  
\label{eq:2}
\end{align}

such that the multiscale tensor paraproduct transforms $T : f \to A(f)$ to 

\begin{align}
\Pi^{(N,N')}_{(A',A'')} : f \to \tilde{A}_{(N,N')}(f) + \Delta_{(N,N')}(A,f)
\label{eq:3}
\end{align}

where $\Delta_{(N,N')}(A,f) = A(f) - \tilde{A}_{(N,N')}(f) \in \Lambda_{2\alpha}([0,1]^2)$ is the residual which has twice the regularity of $f$ and 

\begin{align}
\lVert \Delta_{(N,N')}(A,f) \rVert_{\Lambda_{2\alpha}([0,1]^2)} \leq C_A \lVert f \rVert_{\Lambda_{\alpha}([0,1]^2)}
\label{eq:4}
\end{align}
\label{thm:4.1}
\end{theorem}

Informally, Theorem \ref{thm:4.1} says that if $f$ is a matrix with mixed $\alpha$-H\"older regularity, undergoing a smooth transformation with a $C^2$ function, we can approximate the high frequency content of the matrix, $A(f)$ with equation (\ref{eq:1}) where $P^jP'^{j'}, Q^jQ'^{j'}, Q^jP'^{j'}, P^jQ'^{j'}$ are operators that project $f$ into the tensor scaling, wavelet, wavelet-scaling, and scaling-wavelet bases respectively. \\


Recall from \cite{ankenman2018mixed} the following Besov norm:

\begin{definition}
$\lVert f \rVert_{ (\alpha,p) } = (\sum_{j,k,j',k'} \frac{|<f, \psi^j_k \times \psi^{j'}_{k'}>|^p}{|I^j_k \times I^{j'}_{k'}|^{(\alpha + \frac{1}{2} - \frac{1}{p})p}})^{\frac{1}{p}}$
\label{def:4.2}
\end{definition}

The following corollary characterizes the local regularity gain of the principal term, $\tilde{A}_{(N,N")}(f)$, in theorem \ref{thm:4.1}.

\begin{corollary}
The approximation, $ \tilde{A}_{(N,N')}(f)$, regularizes the range of $T : f \to A(f)$ in the sense of the Besov norm, $\lVert f \rVert_{(\alpha,p)}$, such that $\lVert \tilde{A}_{(N,N')}(f) \rVert_{(\alpha,p)} \leq C_{A,A',A''} N_1N'_1\lVert A(f) \rVert_{(\alpha,p)}$, so long the scaling functions $\phi^j_k, \phi^{j'}_{k'}$ used to build $\tilde{A}_{(N,N')}(f)$ satisfy  $\lVert \phi^j_k \times \phi^{j'}_{k'} \rVert_{L^2} > 1$
\label{cor:4.3}
\end{corollary}

\begin{proof}

First we introduce simplifying notation that permits us to delineate between the scales associated with the approximation, $\tilde{A}_{N,N'}(f)$, and the scales associated with the Besov norm, $\lVert \cdot \rVert_{\alpha,p}$. In particular, we associate the scaling parameters ($j_1,j'_1$) with the approximation $\tilde{A}_{N,N'}(f)$ i.e.

\begin{align}
&  \tilde{A}_{(N_1,N'_1)}(f) =  \sum_{j_1=0}^{N_1} \sum_{j_1'=0}^{N'_1}  A'(P^{j_1}P'^{j'_1}(f)) Q^{j_1}Q'^{j'_1}(f) + A''(P^{j_1}P'^{j'_1}(f))Q^{j_1}P'^{j'_1}(f)P^{j_1}Q'^{j'_1}(f)  
\label{eq:5}
\end{align}

and we associate the scaling parameters, $(j_2, j'_2)$, with the Besov norm, i.e.

\begin{align}
\lVert f \rVert_{ (\alpha,p) } = (\sum_{j_2,k_2,j'_2,k'_2} \frac{|<f, \psi^{j_2}_{k_2} \times \psi^{j'_2}_{k'_2}>|^p}{|I^{j_2}_k \times I^{j'_2}_{k'}|^{(\alpha + \frac{1}{2} - \frac{1}{p})p}})^{\frac{1}{p}}
\label{eq:6}
\end{align}

Suppose $J_2 = (j_2 + j'_2)$ and $J_1 = (j_1 + j'_1)$. The situation where $J_2 \geq J_1$ can be analyzed conceptually; the dyadic rectangle associated with the scales $j_2, j'_2$ under this condition is smaller than the dyadic rectangle associated with the scales $j_1, j'_1$ used in the approximation, so 

\begin{align}
<  A'(P_{k_1}^{j_1}P_{k'_1}'^{j'_1}(f)) Q_{k_1}^{j_1}Q_{k'_1}'^{j'_1}(f) + A''(P_{k_1}^{j_1}P_{k'_1}'^{j'_1}(f))Q_{k_1}^{j_1}P_{k'_1}'^{j'_1}(f)P_{k_1}^{j_1}Q_{k'_1}'^{j'_1}(f) , \psi^{j_2}_{k_2} \times \psi^{j'_2}_{k'_2}> = 0 
\label{eq:7}
\end{align}

for $k_2 = 1, \ldots, 2^{j_2}, k'_2 = 1, \ldots 2^{j'_2}, k_1 = 1, \ldots, 2^{j_1}, k'_1 = 1, \ldots 2^{j'_1}$ for the specified condition of $J_2 \geq J_1$. In the second case, $J_2 < J_1$, the dyadic rectangle associated with are coarser and requires more detailed analysis. Intuitively, the inequality still holds since $\tilde{A}_{(N_1,N'_1)}(f) $ is comprised of averages on dyadic rectangles, so it contains less variation than, $A(f)$, even when $\lVert \cdot \rVert_{ (\alpha,p) }$ is measured with coarser scales. To see this, consider the dyadic rectangles associated with $(j_1,k_1, j'_1,k'_1)$ and $(j_2,k_2, j'_2,k'_2)$. under the specified conditions of $J_2 < J_1$. Without loss of generality, we can proceed by relating $ \lVert f \rVert_{(0,1)}^n \coloneq |<f, \psi^{j_2}_{k_2} \times \psi^{j'_2}_{k'_2}>| $ to the  Besov norm such that $ \lVert f \rVert_{(0,1)}^n $ is equal to $\lVert f \rVert_{ (\alpha,p) }$ modulo the normalizing factor with parameters, $\alpha = 0, p=1$ for the dyadic rectangle associated with $(j_2,k_2, j'_2,k'_2)$. First realize,

\begin{align}
& \lVert A'(P_{k_1}^{j_1}P_{k'_1}'^{j'_1}(f)) Q_{k_1}^{j_1}Q_{k'_1}'^{j'_1}(f) + A''(P_{k_1}^{j_1}P_{k'_1}'^{j'_1}(f))Q_{k_1}^{j_1}P_{k'_1}'^{j'_1}(f)P_{k_1}^{j_1}Q_{k'_1}'^{j'_1}(f)  \rVert_{(0,1)}^n \leq \nonumber \\
& \lVert A'(P_{k_1}^{j_1}P_{k'_1}'^{j'_1}(f)) Q_{k_1}^{j_1}Q_{k'_1}'^{j'_1}(f) \rVert_{(0,1)}^n + \lVert A''(P_{k_1}^{j_1}P_{k'_1}'^{j'_1}(f))Q_{k_1}^{j_1}P_{k'_1}'^{j'_1}(f)P_{k_1}^{j_1}Q_{k'_1}'^{j'_1}(f) \rVert_{(0,1)}^n
\label{eq:8}
\end{align}

by triangle inequality. Let $\chi_{(j_1,k_1,j'_1,k'_1)}(x,y)$ be the characteristic function supported on the dyadic rectangle $I^{j_1}_{k_1} \times  I^{j'_1}_{k'_1}$, $\chi^{-}_{(j_1,k_1)}$ be the characteristic function supported on the left side of the interval of $I^{j_1,k_1}$, and $\chi^{+}_{(j_1,k_1)}$ for the right. Other instances of $\chi$ are defined similarly. Handling the terms on the right side of the inequality separately, beginning with the first term, one obtains

\begin{align}
& \lVert A'(P_{k_1}^{j_1}P_{k'_1}'^{j'_1}(f)) Q_{k_1}^{j_1}Q_{k'_1}'^{j'_1}(f) \rVert_{(0,1)}^n \leq \lVert A'(P_{k_1}^{j_1}P_{k'_1}'^{j'_1}(f)) \rVert_{(0,1)}^n \lVert Q_{k_1}^{j_1}Q_{k'_1}'^{j'_1}(f) \rVert_{(0,1)}^n \nonumber \\
& = \lVert A'(\int_{I^{j_1}_{k_1} \times I^{j'_1}_{k'_1}} \frac{f(x',y') \phi^{j_1}_{k_1}(x') \times \phi^{j'_1}_{k'_1}(y') dx' dy'}{  \lVert \phi^{j_1}_{k_1} \times \phi^{j'_1}_{k'_1} \rVert_{L^2( I^{j_1}_{k_1}  \times I^{j'_1}_{k'_1})} } \chi_{j_1,k_1,j'_1,k'_1}(x,y)) \rVert_{(0,1)}^n \lVert \int_{I^{j_1}_{k_1} \times I^{j'_1}_{k'_1}} \nonumber \\ & \frac{f(x',y') \psi^{j_1}_{k_1}(x') 
\times \psi^{j'_1}_{k'_1}(y') dx' dy' \chi_{(j_1,k_1,j'_1,k'_1)}(x,y)}{ \lVert \psi^{j_1}_{k_1} \times \psi^{j'_1}_{k'_1} \rVert_{L^2( I^{j_1}_{k_1} \times I^{j'_1}_{k'_1} )}} \rVert_{(0,1)}^n \nonumber \\
& = \lVert \frac{A'(\lVert f \rVert_{L^1} \chi_{(j_1,k_1,j'_1,k'_1)}(x,y))}{  \lVert \phi^{j_1}_{k_1} \times \phi^{j'_1}_{k'_1} \rVert_{L^2( I^{j_1}_{k_1} \times I^{j'_1}_{k'_1} )}    } \rVert_{(0,1)}^n \lVert \frac{ ( \lVert f \rVert_{L^1}  (\chi^{-}_{(j_1,k_1)}(x) - \chi^{+}_{(j'_1,k'_1)}(x')) \times (\chi^{-}_{(j_2,k_2)}(y) - \chi^{+}_{(j'_2,k'_2)}(y')) ) }{   \lVert \psi^{j_1}_{k_1} \times \psi^{j'_1}_{k'_1}  \rVert_{L^2( I^{j_1}_{k_1}  \times I^{j'_1}_{k'_1} )}  } \rVert_{(0,1)}^n \nonumber \\
& = < \frac{A'(\lVert f \rVert_{L^1} \chi_{(j_1,k_1,j'_1,k'_1)}(x,y))}{  \lVert \phi^{j_1}_{k_1} \times \phi^{j'_1}_{k'_1} \rVert_{L^2( I^{j_1}_{k_1}  \times I^{j'_1}_{k'_1} )  }}, \psi^{j_2}_{k_2}(x) \times \psi^{j'_2}_{k'_2}(y) > 
\nonumber \\ 
& < (\frac{ \lVert f \rVert_{L^1}  (\chi^{-}_{(j_1,k_1)}(x) - \chi^{+}_{(j'_1,k'_1)}(x)) \times (\chi^{-}_{(j_2,k_2)}(y) - \chi^{+}_{(j'_2,k'_2)}(y)) )} {  \lVert \phi^{j_1}_{k_1} \times \phi^{j'_1}_{k'_1} \rVert_{L^2( I^{j_1}_{k_1} \times I^{j'_1}_{k'_1} )   }  }, \psi^{j_2}_{k_2}(x) \times \psi^{j'_2}_{k'_2}(y) > \nonumber \\
& = \frac{1}{  \lVert \psi^{j_2}_{k_2} \times \psi^{j'_2}_{k'_2} \rVert_{L^2( I^{j_2}_{k_2} \times I^{j'_2}_{k_2} )}^2}  ( \int_{I^{j_2}_{k_2} \times I^{j'_2}_{k'_2}} \frac{A'(\lVert f \rVert_{L^1} \chi_{(j_1,k_1,j'_1,k'_1)}(x,y)) \psi^{j_2}_{k_2}(x') \times \psi^{j'_2}_{k'_2}(y')}{  \lVert \phi^{j_1}_{k_1} \times \phi^{j'_1}_{k'_1} \rVert_{L^2( I^{j_1}_{k_1} \times I^{j'_1}_{k'_1} ) }  }  dx' dy' ) \nonumber \\
&  (\int_{I^{j_2}_{k_2} \times I^{j'_2}_{k'_2}} \frac{ \lVert f \rVert_{L^1} (\psi^{j_1}_{k_1}(x) \times \psi^{j'_1}_{k'_1}(y)) (\psi^{j_2}_{k_2}(x') \times \psi^{j'_2}_{k'_2}(y')) } {  \lVert \phi^{j_1}_{k_1} \times \phi^{j'_1}_{k'_1} \rVert_{L^2( I^{j_1}_{k_1} \times I^{j'_1}_{k'_1} )   }  } dx' dy')  \psi^{j_2}_{k_2}(x) \times \psi^{j'_2}_{k'_2}(y) \nonumber \\
& = \frac{ \lVert A'(\lVert f \rVert_{L^1} \chi_{(j_1,k_1,j'_1,k'_1)}(x,y)) \psi^{j_2}_{k_2}(x') \times \psi^{j'_2}_{k'_2}(y')   \rVert_{L^1( I^{j_2}_{k_2} \times I^{j'_2}_{k'_2})}  \lVert  \lVert f \rVert_{L^1} (\psi^{j_1}_{k_1}(x) \times \psi^{j'_1}_{k'_1}(y)) (\psi^{j_2}_{k_2}(x') \times \psi^{j'_2}_{k'_2}(y'))   \rVert_{L^1( I^{j_2}_{k_2} \times I^{j'_2}_{k'_2})} }{\lVert \psi^{j_2}_{k_2} \times \psi^{j'_2}_{k'_2} \rVert_{L^2( I^{j_2}_{k_2} \times I^{j'_2}_{k'_2} )}^2  \lVert \phi^{j_1}_{k_1} \times \phi^{j'_1}_{k'_1} \rVert_{L^2( I^{j_1}_{k_1} \times I^{j'_1}_{k'_1})}^2}  \nonumber \\
& \leq  \frac{ \lVert A'(\lVert f \rVert_{L^1} \chi_{(j_1,k_1,j'_1,k'_1)}(x,y)) \psi^{j_2}_{k_2}(x') \times \psi^{j'_2}_{k'_2}(y')   \rVert_{L^1( I^{j_2}_{k_2} \times I^{j'_2}_{k'_2})}  \lVert  \lVert f \rVert_{L^1} (\psi^{j_1}_{k_1}(x) \times \psi^{j'_1}_{k'_1}(y)) (\psi^{j_2}_{k_2}(x') \times \psi^{j'_2}_{k'_2}(y'))   \rVert_{L^1( I^{j_2}_{k_2} \times I^{j'_2}_{k'_2})} }{\lVert \psi^{j_2}_{k_2} \times \psi^{j'_2}_{k'_2} \rVert_{L^2( I^{j_2}_{k_2} \times I^{j'_2}_{k'_2} )}^2  } \nonumber \\
& \leq C_{A'} \frac{ \lVert f \rVert_{L^1( I^{j_2}_{k_2} \times I^{j'_2}_{k'_2}    )}^2  }{ \lVert  \psi^{j_2}_{k_2} \times \psi^{j'_2}_{k'_2} \rVert_{L^2( I^{j_2}_{k_2} \times I^{j'_2}_{k'_2} )}^2} \nonumber \\
\label{eq:9}
\end{align}

where the first inequality comes from Cauchy-Schwartz, the second identity from the definition of $\lVert \cdot \rVert_{(0,1)}^n$, the the next inequality from, $ \lVert \phi^{j_1}_{k_1} \times \phi^{j'_1}_{k'_1} \rVert_{L^2( I^{j_1}_{k_1} \times I^{j'_1}_{k'_1})}^{-2}  $, and the last inequality from the norms of the composition, $A'(f)$. Since the second term has the identity:

\begin{align}
\lVert A''(P_{k_1}^{j_1}P_{k'_1}'^{j'_1}(f))Q_{k_1}^{j_1}P_{k'_1}'^{j'_1}(f)P_{k_1}^{j_1}Q_{k'_1}'^{j'_1}(f)   \rVert_{(0,1)}^n = \lVert A''(P_{k_1}^{j_1}P_{k'_1}'^{j'_1}(f))Q_{k_1}^{j_1}Q_{k'_1}'^{j'_1}(f)\rVert_{(0,1)}^n
\label{eq:10}
\end{align}

the estimate is identical to that of the first term modulo a constant related to $A^{''}$, giving:

\begin{align}
\lVert A''(P_{k_1}^{j_1}P_{k'_1}'^{j'_1}(f))Q_{k_1}^{j_1}Q_{k'_1}'^{j'_1}(f)\rVert_{(0,1)}^n \leq C_{A^{''}}  \frac{ \lVert f \rVert_{L^1( I^{j_2}_{k_2} \times I^{j'_2}_{k'_2}    )}^2  }{ \lVert  \psi^{j_2}_{k_2} \times \psi^{j'_2}_{k'_2} \rVert_{L^2( I^{j_2}_{k_2} \times I^{j'_2}_{k'_2} )}^2}
\label{eq:11}
\end{align}

Combining both estimates gives, 

\begin{align}
& \lVert A'(P_{k_1}^{j_1}P_{k'_1}'^{j'_1}(f)) Q_{k_1}^{j_1}Q_{k'_1}'^{j'_1}(f) + A''(P_{k_1}^{j_1}P_{k'_1}'^{j'_1}(f))Q_{k_1}^{j_1}P_{k'_1}'^{j'_1}(f)P_{k_1}^{j_1}Q_{k'_1}'^{j'_1}(f)  \rVert_{(0,1)}^n \leq C_{A',A''} \frac{ \lVert f \rVert_{L^1( I^{j_2}_{k_2} \times I^{j'_2}_{k'_2}    )}^2  }{ \lVert  \psi^{j_2}_{k_2} \times \psi^{j'_2}_{k'_2} \rVert_{L^2( I^{j_2}_{k_2} \times I^{j'_2}_{k'_2} )}^2} \nonumber \\
& = C_{A,A',A''} \lVert A(f) \rVert_{(0,1)}^n
\label{eq:12}
\end{align}

Recall that $\lVert \cdot \rVert_{(0,1)}^n$ is $\lVert \cdot \rVert_{(\alpha,p)}$ for fixed scales $(j'_2,j_2)$, $\alpha=0,p=1$, modulo a normalizing factor related to the support of the tensor wavelets. By inequalities \ref{eq:7} and \ref{eq:11} one has 

\begin{align}
\lVert A'(P^{j_1}P'^{j'_1}(f)) Q^{j_1}Q'^{j'_1}(f) + A''(P^{j_1}P'^{j'_1}(f))Q^{j_1}P'^{j'_1}(f)P^{j_1}Q'^{j'_1}(f) \rVert_{(\alpha,p)} \leq C_{A,A',A''} \lVert A(f) \rVert_{(\alpha,p)}
\label{eq:13}
\end{align}

for $j_1 = 0, \ldots, N_1, j'_1 = 0, \ldots, N'_1$ since we've shown the approximation at the fixed scales $(j_1,j'_1)$ holds for both conditions $J_2 \geq J_1$ and $J_2 < J_1$. We then have 

\begin{align}
& \lVert \tilde{A}_{(N_1,N'_1)}(f) \rVert_{(\alpha,p)} \nonumber \\  
& = \lVert \sum_{j_1=0}^{N_1} \sum_{j_1'=0}^{N'_1}  A'(P^{j_1}P'^{j'_1}(f)) Q^{j_1}Q'^{j'_1}(f) + A''(P^{j_1}P'^{j'_1}(f))Q^{j_1}P'^{j'_1}(f)P^{j_1}Q'^{j'_1}(f)  \rVert \nonumber \\
& \leq \sum_{j_1=0}^{N_1} \sum_{j_1'=0}^{N'_1} \lVert A'(P^{j_1}P'^{j'_1}(f)) Q^{j_1}Q'^{j'_1}(f) + A''(P^{j_1}P'^{j'_1}(f))Q^{j_1}P'^{j'_1} \rVert \nonumber \\
& \leq C_{A,A',A''} N_1N'_1\lVert A(f) \rVert_{(\alpha,p)}
\label{eq:14}
\end{align}

by triangle inequality, completing the proof.

\end{proof}

\begin{table}[h]
  \centering
  \caption{Nomenclature}
  \label{tab:yourlabel}
  \begin{tabular}{lcc}
    \toprule
    Object &  Name  \\
    \midrule
    $\Pi^{(N,N')}_{(A',A'')} : A(f) \to \tilde{A}(f) + \Delta(A,f) $ & Tensor paraproduct \\
    $\tilde{A}_{(N,N')}(f) + \Delta_{(N,N')}(A,f)$ & Tensor paraproduct decomposition (TPD) \\
    $\tilde{A}_{(N,N')}(f)$ & Tensor paraproduct approximation (TPA) \\
    $\Delta_{(N,N')}(A,f)$ & Tensor paraproduct residual (TPR)  \\
    \bottomrule
  \end{tabular}
  \label{tab:1}
\end{table}

Table \ref{tab:1} is included to provide clarity on various terminology used throughout the paper. In particular we distinguish between the operator used to construct the approximation, the range of the operator, and the principal and residual components of the range.

\section{Quasilinearization of kernels of integral operators}

We are focused on the quasilinearization of kernels, $K(x,y)$, of integral operators:

\begin{align}
T : f(y) \to \int K(x,y)f(y) dy
\label{eq:15}
\end{align}

There are three situations in which the quasilinearization is useful. The first is a direct application of the results of \cite{fasina2025quasilinearization}, which comes from realizing that a family of kernels of integral operators are nonlinear compositions of some affinity, $d(x,y)$. This realization leads to enhanced local and global regularity of $K(x,y)$ (as discussed in section 4). In the case of the potential kernel, $K_n(x,y)$, one also gains rapid off-diagonal decay of the principal component (the TPA) of the decomposition, permitting the range of $T$ to be computed rapidly. We refer to this as the \emph{direct} quasilinearization method. The second situation follows by exploiting that regularity gains are an immediate consequence of the tensor paraproduct decomposition, by virute of the results of \cite{fasina2025quasilinearization}. Thus following application of a $C^2(\mathbb{R})$ scalar function to a kernel and computing the tensor paraproduct decomposition yields regularity gains at the cost of an error related to the difference between the original kernel and the $C^2(\mathbb{R})$ scalar function applied to the kernel. We refer to this as the \emph{indirect} quasilinearization method. Finally, in situations where the expansion coefficients for a distance, $d(x,y)$, are available and one needs to rapidly construct the kernel, the tensor paraproduct formula permits us to multiply the high-low paraproducts, $A'(P^jP'^{j'}(d)), A''(P^jP'^{j'}(d))$, to the expansion coefficients of the distance, $d(x,y) = \sum_{j,j'=0}^{N,N'} Q^jQ'^{j'}(d) +  P^jQ'^{j'}(d)Q^jP'^{j'}(d)$, to construct the kernel. We elaborate on the \emph{direct} and \emph{indirect} quasilinearization with formal mathematical statements below.

\subsection{Direct quasilinearization}

Let the kernel, $K(x,y)\coloneq A_1(d(x,y))$, be understood by the following operator:

\begin{align}
S_1 : d(x,y) \to A_1(d(x,y))
\label{eq;16}
\end{align}

The following proposition permits us to replace $S_1$ with the tensor paraproduct, $\Pi^{N,N'}_{A_1' A_1^{''}}$:

\begin{proposition}
Suppose $d(x,y) \in \Lambda_{\alpha}([0,1]^2), A_1 \in C^2(\mathbb{R})$ then the operator $S_1 : d(x,y) \to A_1(d(x,y))$ can be replaced by the tensor paraproduct ${\Pi}^{N,N'}_{A_1', A_1^{''}}$:

\begin{align}
\Pi^{N,N'}_{A_1' A_1^{''}} : d(x,y) \to \tilde{A}_{1(N,N')}(d(x,y)) + \Delta_{(N,N')}(A_1, d(x,y))
\label{eq:17}
\end{align}

\label{prop:5.1}
\end{proposition}

\begin{proof}
Immediate from the regularity assumptions on $d(x,y)$ and $A_1$.
\end{proof}

\begin{remark}
Proposition \ref{prop:5.1} is a formal statement permitting one to apply the tensor paraproduct decomposition, under appropriate regularity conditions 
\label{rem:5.2}
\end{remark}

The local and global regularity gains of the range of the tensor paraproduct are characterized by the following theorem. 

\begin{theorem}
The TPA, $\tilde{A}_{1(N,N")}(d(x,y))$, contains more regularity in the sense of the Besov norm, $\lVert \cdot \rVert_{\alpha,p}$, than the range of $S_1$ (the original kernel, $A_1(d))$. In particular, one has:

\begin{align}
& \lVert \tilde{A}_{1(N,N')}(d(x,y)) \rVert_{\alpha,p} \leq \lVert A_1(d(x,y)) \rVert_{\alpha,p} \nonumber \\
\label{eq:18}
\end{align}

Furthermore, the TPR, $\Delta_{N,N'}(A,d(x,y))$, contains more global regularity than than the original kernel, $A(d(x,y))$, since $\Delta_{N,N'}(A,d(x,y)) \in \Lambda_{2\alpha}$ while $A(f) \in \Lambda_\alpha$

\label{thm:5.3}
\end{theorem}

\begin{proof}
Immediate consequence of corollary \ref{cor:4.3} and theorem \ref{thm:4.1}
\end{proof}

\subsection{Indirect quasilinearization}

Now consider the situation where we apply a smooth function, $A_2$ to a kernel, $K(x,y)$. The following proposition permits us to replace the operator

\begin{align}
S_2 : K(x,y) \to A_2(K(x,y))
\label{eq:19}
\end{align}

with the tensor paraproduct, $\Pi^{N,N'}_{A_2' A_2^{''}}$.

\begin{proposition}
If $K(x,y) \in \Lambda_{\alpha}([0,1]^2), A_2 \in C^2(\mathbb{R})$ the operator $S_2 : K(x,y) \to A_2(K(x,y))$ can be replaced by the tensor paraproduct ${\Pi}^{N,N'}_{A_2', A_2^{''}}$:

\begin{align}
\Pi^{N,N'}_{A_2' A_2^{''}} : K(x,y) \to \tilde{A}_{2(N,N')}(K(x,y)) + \Delta_{(N,N')}(A_2, K(x,y))
\label{eq:20}
\end{align}

\label{prop:5.4}
\end{proposition}

\begin{proof}
Immediate from the regularity assumptions on $K(x,y)$ and $A_2$.
\end{proof}

\begin{remark}
Note the distinction between Propositions \ref{prop:5.1} and \ref{prop:5.4}. In the former situation, we split the kernel $K(x,y)$ of the integral operator, $T$, into $A_1$ and $d(x,y)$ while in the latter situation, we consider the application of $A_2$ to $K(x,y)$
\label{rem:5.5}
\end{remark}

Let $K_{A_2}(x,y) \coloneq A_2(K(x,y))$. The advantage of $K_{A_2}(x,y)$ is the gain in regularity, in the sense of the Besov norm, it holds over $K(x,y)$ which is formalized as:

\begin{theorem}
        
The range of the tensor paraproduct, $\Pi^{N,N'}_{A_2' A_2^{''}}$, contains more local regularity (in the sense of the Besov norm, $\lVert \cdot \rVert_{\alpha,p}$) and global $\alpha$-Holder regularity than the original kernel, $K(x,y)$.

\label{thm:5.6}
\end{theorem}

\begin{proof}
Immediate consequence of corollary \ref{cor:4.3} and theorem \ref{thm:4.1}
\end{proof}

\section{Algorithm description and complexity analysis}

\subsection{Algorithm description}

Algorithms \ref{alg:1} and \ref{alg:2} below outline (1) how to construct the TPA in $O(\frac{1}{\epsilon}log(\frac{1}{\epsilon}))$ where $\epsilon$ is the precision specifying the smallest dyadic rectangle following \cite{coifman2011harmonic} and (2) an algorithm for compressing the TPA when off-diagonal decay is present, following \cite{beylkin1991fast}.

\begin{algorithm}[H]
\caption{Compression of coefficient matrices in $O(\frac{1}{\epsilon}log(\frac{1}{\epsilon}))$ for prrecision $\epsilon$}
\label{alg:1}
\begin{algorithmic}[1]
\REQUIRE $ \epsilon \coloneq 2^{-m}, A \in C^2(\mathbb{R}), d(x,y) \in \Lambda_{\alpha}([0,1]^2), \{ \{ \phi^j_k , \psi^j_k \}_{j,k} , \{ \phi^{j'}_{k'} , \psi^{j'}_{k'} \}_{j',k'}  : j + j' = m  \}$
\ENSURE $ \alpha^{m}(d) \coloneq \sum_{ \{ j,j' : m = j + j' \} } (Q^jQ'^{j'}(d)) , \gamma^{m}(d) \coloneq \sum_{ \{ j,j' : m = j + j' \} } (P^jQ'^{j'}(d))$, \\
$  \beta^{m}(d) \coloneq \sum_{ \{ j,j' : m = j + j' \} } (Q^jP'^{j'}(d)), \eta^{m}(d) \coloneq \sum_{ \{ j,j' : m = j + j' \} } (P^jP'^{j'}(d))  $
\FOR{ $j + j' = m$} 
\FOR{ $k = 1, 2, \ldots, 2^j, k' = 1, 2, \ldots, 2^{j'}$}
\STATE $Q^j_{k}Q'^{j'}_{k'}(d) = (<d(x,x'), \psi^j_k(x) \times \psi^{j'}_{k'}(x')>) \chi^{j,j'}_{k,k'}(y,y')$
\STATE $P^j_{k}Q'^{j'}_{k'}(d) =  (<d(x,x'), \phi^j_k(x) \times \psi^{j'}_{k'}(x')>) \chi^{j,j'}_{k,k'}(y,y')$
\STATE $Q^j_{k}P'^{j'}_{k'}(d) =   (<d(x,x'), \psi^j_k(x) \times \phi^{j'}_{k'}(x')>) \chi^{j,j'}_{k,k'}(y,y')$
\STATE $P^j_{k}P'^{j'}_{k'}(d) =  (<d(x,x'), \phi^j_k(x) \times \phi^{j'}_{k'}(x')>) \chi^{j,j'}_{k,k'}(y,y')$
\ENDFOR
\STATE $ Q^{j}Q'^{j'}(d) = \sum_{k,k'=1}^{2^j,2^{j'}} Q^j_{k}Q'^{j'}_{k'}(d), P^{j}Q'^{j'}(d) = \sum_{k,k'=1}^{2^j,2^{j'}} Q^j_{k}Q'^{j'}_{k'}(d), Q^{j}P'^{j'}(d) = \sum_{k,k'=1}^{2^j,2^{j'}} Q^j_{k}Q'^{j'}_{k'}(d), P^{j}P'^{j'}(d) = \sum_{k,k'=1}^{2^j,2^{j'}} Q^j_{k}Q'^{j'}_{k'}(d)$
\ENDFOR
\end{algorithmic}
\end{algorithm}

\begin{algorithm}[H]
\caption{Compression of TPA by discarding coefficients less than $\delta > 0$ for $O(M)$ storage}
\label{alg:2}
\begin{algorithmic}[1]
\REQUIRE $ \delta > 0, \alpha^{m}(d) , \gamma^{m}(d) , \beta^{m}(d), \eta^{m}(d),N_x,N_y$
\ENSURE $ \tilde{A}(f)_{\delta,m} = A'(\eta^{m}(d))\alpha^{m}(d) + A''(\eta^{m}(d))\beta^{m}(d)\gamma^{m}(d) $
\FOR{$i=1,2, \ldots, M_x$} 
\FOR{$l = 1, 2, \ldots, M_y$}
\STATE $\eta^{m}(A',f)[l,i] \leftarrow A'(\eta^{m}(d)[l,i])$, $\eta^{m}(A'',d)[l,i] \leftarrow A''(\eta^{m}(d)[l,i])$
\ENDFOR
\ENDFOR
\STATE $ \tilde{A}(f)_{m} = (\eta^{m}(A',f)[l,i] \odot \alpha^{m}(f)) + (\eta^{m}(A'',f)[l,i] \odot \beta^{m}(f) \odot \gamma^{m}(f)) $
\FOR{$i=1,2, \ldots, M_x$} 
\FOR{$l = 1, 2, \ldots, M_y$}
\IF{ $\texttt{abs}(\tilde{A}(f)_{m}[l,i]) < \delta$ }
\STATE $\tilde{A}(f)_{m}[l,i] = 0 $
\ENDIF
\ENDFOR
\ENDFOR
\STATE $ \tilde{A}(f)_{\delta,m} \leftarrow \tilde{A}(f)_{m} $
\end{algorithmic}
\end{algorithm}

Algorithm 1 is useful in all scenarios where the direct quasilinearization described in section 5 is applied. In algorithm 2, $M_x$ and $M_y$ are the number of samples in the $x$ and $y$ directions; algorithm 2 is useful for situations where the TPA has rapid off-diagonal decay. Recall that we wish to compute the range of $T: f(y) \to \int K(x,y)f(y)$ using the approximation $\tilde{A}(f)_{\delta,m}$ for the kernel, $K(x,y)$. This amounts to computing

\begin{align}
& \int K(x,y)f(y) dy \simeq \int \tilde{A}(d)_{\delta,m}(x,y) + \Delta(A,d)(x,y) f(y) dy \nonumber \\
& \simeq (A'(\eta^{m}(d))\alpha^{m}(d)  +   A''(\eta^{m}(d))\beta^{m}(d)\gamma^{m}(d) + \Delta(A,d)(x,y)) f(y) \nonumber \\
& = \tilde{g}(x)
\label{eq:21}
\end{align}

\subsection{Complexity analysis}

We briefly state informal complexity results which build on existing methods for fast construction of mixed-Holder functions from \cite{coifman2011harmonic} and the FWT method \cite{beylkin1991fast} which exploits rapid off-diagonal decay of kernels by compression of wavelet coefficients.

\begin{enumerate}
    \item Fast construction of $\tilde{A}(f)$: If $A(f) \in \Lambda_{\alpha}([0,1]^2)$, $A(f)$ can be computed in $O(\frac{1}{\epsilon}log(\frac{1}{\epsilon}))$ where $\epsilon$ is the specified precision of a dyadic rectangle \cite{coifman2011harmonic}.
    \item Fast matrix-vector multiplication: thresholding coefficients of $\tilde{A}(f)_m$ less than $\delta$, as seen in algorithm 2, leads to $O(M_x)$ storage, following the FWT method outlined in \cite{beylkin1991fast}
\end{enumerate}

\section{Examples}

Here we discuss two commonly encountered objects in computational physics and mathematics: the potential and fractional Cauchy kernels:

\begin{align}
K_n(x,y) \coloneq log(\sum_{k=0}^n a_k d(x,y)^{-k} ) \qquad K_{\alpha}(r,\theta) \coloneq (re^{i \theta} - z)^{-\alpha} 
\label{eq:22}
\end{align}

We apply the direct and indirect quasilinearization techniques for the potential and fractional Cauchy kernels, respectively and include computational examples for the potential kernel.  

\subsection{Direction quasilinearization: potential kernel}

Consider the potential kernel, $K_n(x,y)$, described above under the conditions of $a_k = 1$ for $ 0 \leq k \leq n $. Consider the scalar function, $\varphi_1$:

\begin{align}
\varphi_1 : f \to log(f)
\label{eq:23}
\end{align}

Then one can understand the potential kernel, $K_n(x,y)$, through the operator $S_{\rho_1}$:

\begin{align}
S_{\rho_1} : d(x,y)^{-(n+1)} \to \varphi_1(d(x,y)^{-(n+1)})
\label{eq:24}
\end{align}

which can be replaced with the tensor paraproduct 

\begin{align}
\Pi^{N,N'}_{\varphi'_1 \varphi^{''}_1} : d(x,y)^{-(n+1)} \to \tilde{\varphi}_{1_{(N,N')}}(d(x,y)^{-(n+1)}) + \Delta_{(N,N')}(\tilde{\varphi}_1, d(x,y)^{-(n+1)})
\label{eq:25}
\end{align}

by appealing to proposition \ref{prop:5.1}. Computing $\tilde{\varphi}_1(d(x,y)^{-(n+1)})$ (where we suppress $d(x,y)$ such that $d\coloneq d(x,y)$) gives:

\begin{align}
\tilde{\varphi}_1(d(x,y)^{-(n+1)}) \coloneq \sum_{j,j'=0}^{N,N'} P^jP'^{j'}(d^{-(n+1)})^{-1}Q^jQ'^{j'}(d^{-(n+1)}) - P^jP'^{j'}(d^{-(n+1)})^{-2}Q^jP'^{j'}(d^{-(n+1)})P^jQ'^{j'}(d^{-(n+1)})
\label{eq:26}
\end{align}

\begin{remark}
The scalar function, $\varphi_1$, permits one to replace the log function with its derivatives. This is crucial since the off-diagonal decay of $K_n(x,y)$ is slow when the support of log is greater than 1. Replacing $S_{\varphi_1}$ with  $\Pi^{N,N'}_{\varphi'_1 \varphi^{''}_1}$ allows one to exploit the properties of the fast off-diagonal decay of $d^{-(n+1)}$, leading to matrix compression.
\label{rem:7.1}
\end{remark}

\subsection{Indirect quasilinearization: fractional Cauchy kernel}

Consider the scalar function $\varphi_2 : f \to f^2$. Pointwise application of $\varphi_2$ to $K_{\alpha}(r,\theta)$ results in a new kernel which can be understood by the operator: $S_{\varphi_2}$ 

\begin{align}
S_{\varphi_2} : K_{\alpha}(r,\theta) \to \varphi_2(K_{\alpha}(r,\theta))
\label{eq:27}
\end{align}

Direct quasilinearization of $S_{\varphi_2}$ (and indirect quasilinearization of $K_{\alpha}(r,\theta)$) permits one to replace $S_{\varphi_2}$ with the tensor paraproduct 

\begin{align}
\Pi^{N,N'}_{\varphi'_2 \varphi^{''}_2} : K_\alpha(r,\theta) \to \tilde{\varphi}_2(K_\alpha(r,\theta)) + \Delta_{(N,N')}(\tilde{\varphi}_2, K_\alpha(r,\theta))
\label{eq:28}
\end{align}

since $\varphi_2 \in C^2(\mathbb{R})$ and $K_\alpha(r,\theta) \in \Lambda_{\alpha}$. computing the principal term, $\tilde{\varphi}_2(K_\alpha(r,\theta))$, yields

\begin{align}
\tilde{\varphi}_2(K_\alpha(r,\theta)) = \sum_{j,j'=0}^{N,N'} 2 P^jP'^{j'}(K_\alpha(r,\theta) )Q^jQ'^{j'}(K_\alpha(r,\theta)) + Q^jP'^{j'}(K_\alpha(r,\theta)) P^jQ'^{j'}(K_\alpha(r,\theta))
\label{eq:29}
\end{align}

\begin{remark}
The indirect quasilinearization of $K_{\alpha}(r,\theta)$ leads to enhanced regularity ( since $\Delta_{(N,N')}(\tilde{\varphi}_2, K_\alpha(r,\theta)) \in \Lambda_{2\alpha}$ by Theorem \ref{thm:4.1}) with a cost associated with inverting $\varphi_2$.
\label{rem:7.2}
\end{remark}

\section{Numerical results: regularization and compression of potential kernel with direct quasilinearization}

Here we apply the direct quasilinearization method to the potential kernel, $K_n(x,y)$. We (1) describe how the data was generated and under what conditions (2) numerically verify that the TPA, $\tilde{\varphi}_{1_{(N,N')}}(d(x,y)^{-(n+1)})$, has more local regularity than the original potential kernel, $K_n(x,y)$, in the sense of the Besov norm, and (3) showcase results regarding the compression of TPA, $\tilde{\varphi}_{1_{(N,N')}}(d(x,y)^{-(n+1)})$.

\subsection{Raw data and tensor paraproduct decomposition}

For all experiments we consider a potential kernel which is the logarithm of a degree 5 polynomial i.e. $K_{5}(d(x,y))$. The degree 5 polynomial is used to avoid the numerical difficulties associated with selecting the appropriate class of vanishing moments of wavelets which to ensure sufficient off-diagonal decay following direct quasilinearization of $K_{5}(d(x,y))$. To build $K_{5}(d(x,y))$, we first construct the affinity, $d(x,y) : [0,1]^2 \to \lVert \vec{x} - \vec{y} \rVert_{l^2([0,1]^2)}$ where $\vec{x} = (x_1,x_2)$ and $\vec{y} = (y_1,y_2)$. We sample $M$ points, $( \vec{x}_i , \vec{y}_i \}_{i=1}^{M}$, to match the dimension of the tensor we wish to construct. Here we only consider $M = \{ 128, 256, 512 \}$. The data generation process for our experiments are described below in Figure 1.

\begin{figure}[H]
    \centering
    \includegraphics[width=0.6\textwidth]{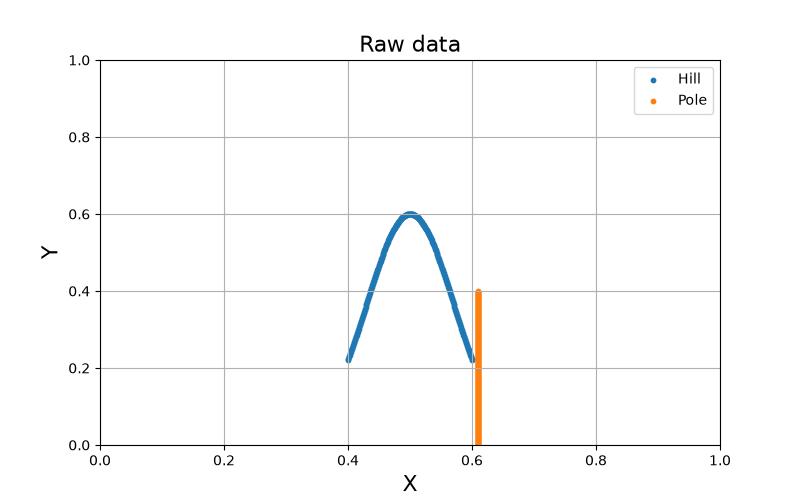}
    \caption{Data generation $\{ \vec{x}_i\}_{i=1}^{256}, \{ \vec{y}_i \}_{i=1}^{256} \in [0,1]^2$ generated for the potential kernel with $M = 256$ points. The blue bell curve is generated with the coordinates $( x_1(j) = 0.4 + j\frac{0.2}{256} ,x_2 = e^{-(\frac{x_1(j) - 0.5}{0.1})^2}0.6)$ where $j = 0 ,\ldots, 255$. The orange pole is generated with coordinates $(y_1 = 0.61 , y_2(j) = j\frac{0.4}{256}), j=0,\ldots, 255$.}
    \label{fig:1}.
\end{figure}

The support of $\tilde{\varphi_1}$ is then computed by collapsing the geometric series, $\sum_{n=0}^{5} d(x,y) \simeq d(x,y)^{-5}$. The approximation is then computed in $O(\frac{1}{\epsilon}log(\frac{1}{\epsilon}))$ time using algorithm \ref{alg:1}, where $\epsilon = 2^{-(j_1 + j'_1)}$, where $ j_1, j'_1 \in \mathbb{N}$ are the scaling parameters associated with the decomposition and $\epsilon$ is the precision at which we compute the approximation. Since the approximation can only be computed if $d(x,y)^{-24} \in \Lambda_{\alpha}([0,1]^2)$ (by the regularity assumptions of Theorem \ref{thm:4.1}), we illustrate this phenomenon by showing the wavelet coefficients of $d(x,y)^{-5}$ decay exponentially. In particular, we compute the wavelet expansion of coefficients of $d(x,y)^{-5}$ with $d(x,y)^{-5} = \sum_{j,j',k,k'} \alpha^{j,j'}_{k,k'} \psi^j_k(x) \times \psi^{j'}_{k'}(x')$ and plot $|\alpha^{j,j'}_{k,k'} |$ verify the decay satisfies $ | \alpha^{j,j'}_{k,k'} | \leq 2^{-(j + j')} $ \cite{stephane1999wavelet}. The TPA, $\tilde{\varphi}_{1_{(N,N')}}(d(x,y)^{-(n+1)})$, is consequently computed with the formula 

\begin{align}
& \tilde{\varphi}_{1_{J_1}}(d(x,y)^{-(n+1)}) = \sum_{j_1,j_1' : j_1 + j'_1 = J_1} P^{j_1}P'^{j_1'}(d^{-(n+1)})^{-1}Q^{j_1}Q'^{j_1'}(d^{-(n+1)}) - \nonumber \\
& P^{j_1}P'^{j_1'}(d^{-(n+1)})^{-2}Q^{j_1}P'^{j_1'}(d^{-(n+1)})P^{j_1}Q'^{j_1'}(d^{-(n+1)})
\label{eq:30}
\end{align}

where $J_1$ specifies the size of the dyadic rectangles satisfying precision $\epsilon = 2^{-J_1}$, and $Q^{j_1}Q'^{j_1'}, P^{j_1}Q'^{j_1'}, Q^{j_1}P'^{j_1'}$ and $P^{j_1}P'^{j_1'}$ are identified with projections into the tensor wavelet, scaling-wavelet, wavelet-scaling, and scaling bases for all dyadic rectangles satisfying precision of $\epsilon = 2^{J_1}$. We identify $\tilde{\varphi}_{1_{(N,N')}}(d(x,y)^{-(n+1)}) $ from the previous section with $\tilde{\varphi}_{1_{J_1}}(d(x,y)^{-(n+1)}) $ to denote the precision level.

\begin{figure}[H]
    \centering
    \includegraphics[width=0.95\textwidth]{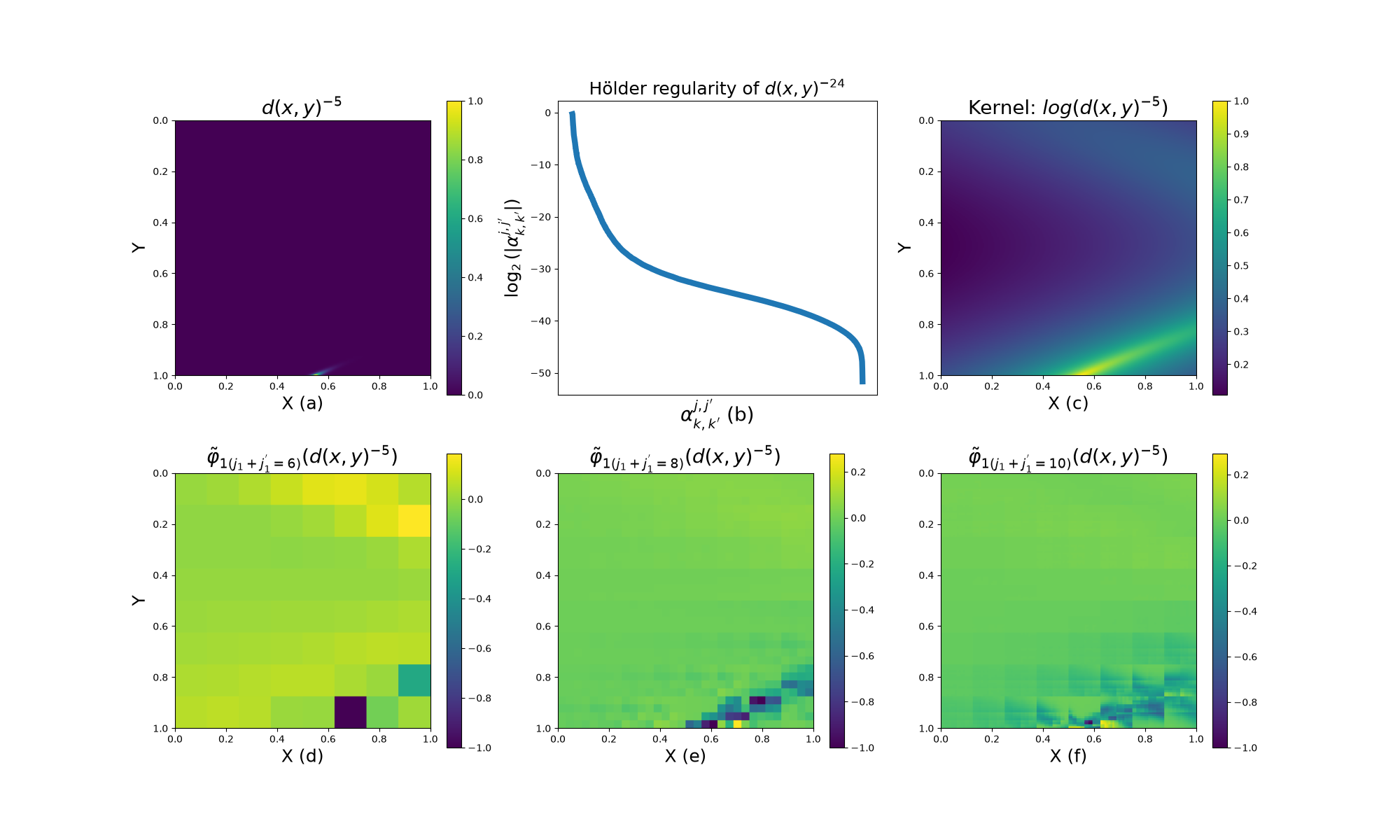}
    \caption{Top row: left to right; (a) 5th degree polynomial of the distance matrix, $d(x,y)^{-5}$ (b) log-scale plot of wavelet coefficients, $\alpha^{j,j'}_{k,k'}$ of $d(x,y)^{-5}$ (c) original potential kernel, $K_{5}(d(x,y))$. Bottom row: left to right; (d) TPA of $K_{5}(d(x,y))$ for precision $\epsilon = 2^{-6}$ (e) TPA of $K_{5}(d(x,y))$ for precision $\epsilon = 2^{-8}$ (f) TPA of $K_{5}(d(x,y))$ for precision $\epsilon = 2^{-10}$.
    Subplots (a), (c) , (d-f) are normalized relative to the maximum absolute value. The rapid off-diagonal decay for $d(x,y)^{-5}$ is seen in (a) while the slow off-diagonal decay is seen in (c), as expected since the range of $d(x,y)^{-5} > 1$ In (b) we see $d(x,y)^{-5} \in \Lambda_{\alpha}$ since linear decay is seen on the log-scale plot. The resolution of the TPA increases as one goes from left to right on the bottom row (d-f), which is expected since the precision becomes more fine scale. }
    \label{fig:2}
\end{figure}

\subsection{Local regularity}

The Besov norm (see definition \ref{def:4.2}) can be used to determine the local regularity gains of $\tilde{\varphi}_{1_{J_1}}(d(x,y)^{-5})$ by measuring the relative variation. We expect the TPA, $\tilde{\varphi}_{1_{J_1}}(d(x,y)^{-5})$, to have increased local regularity, for analytical reasons (discussed in corollary \ref{cor:4.3}) and conceptual reasons (the TPA is constructed partially by averaging over dyadic rectangles). We use Table 2 to numerically verify our theoretical findings; in particular, we measure the relative variation by computing the relative Besov norm of the TPA to the original kernel, $K_{5}(d(x,y))$. Values less than 1 are expected for all conditions since lower values of $\lVert \cdot \rVert_{(\alpha,p)}$ correspond to more regularity. 

\begin{table}[H]
\centering
\caption{  Relative variation $ \lVert \tilde{\varphi}_{1_{J_1}}(d(x,y)^{5}) \rVert_{ (\alpha,1) } / \lVert A(f) \rVert_{ (\alpha,1) }   $ }
\begin{tabular}{c|ccc}
\toprule
 & $\alpha = 5e-3$ & $\alpha = 5e-2$ & $\alpha = 5e-1$ \\
\midrule
 & \multicolumn{3}{c}{M = 128} \\
\midrule
$ \epsilon = 7.81 \times 10^{-3} $& $4.46 \times 10^{-1}$ & $4.74 \times 10^{-1}$ & $4.89 \times 10^{-1}$ \\
$ \epsilon = 3.91 \times 10^{-3} $& $4.32 \times 10^{-1}$ & $4.71 \times 10^{-1}$ & $4.96 \times 10^{-1}$ \\
$ \epsilon = 1.95 \times 10^{-3} $& $2.91 \times 10^{-1}$ & $4.17 \times 10^{-1}$ & $5.60 \times 10^{-1}$ \\
\midrule
 & \multicolumn{3}{c}{M = 256} \\
\midrule
$ \epsilon = 7.81 \times 10^{-3} $ & $4.04 \times 10^{-1}$ & $5.88 \times 10^{-1}$ & $6.30 \times 10^{-1}$ \\
$ \epsilon = 3.91 \times 10^{-3} $ & $3.93 \times 10^{-1}$ & $5.82 \times 10^{-1}$ & $6.38 \times 10^{-1}$ \\
$ \epsilon = 9.77 \times 10^{-4} $ & $2.72 \times 10^{-1}$ & $4.99 \times 10^{-1}$ & $7.11 \times 10^{-1}$ \\
\midrule
 & \multicolumn{3}{c}{M = 512} \\
\midrule
$ \epsilon = 1.95 \times 10^{-3} $ & $4.52 \times 10^{-1}$ & $8.21 \times 10^{-1}$ & $9.81 \times 10^{-1}$ \\
$ \epsilon = 9.77 \times 10^{-4} $ & $4.38 \times 10^{-1}$ & $8.09 \times 10^{-1}$ & $9.83 \times 10^{-1}$ \\
$ \epsilon =  4.88 \times 10^{-4} $ & $2.93 \times 10^{-1}$ & $6.67 \times 10^{-1}$ & $9.84 \times 10^{-1}$ \\
\hline
\end{tabular}
\caption{ Each row represents a different precision level given by $\epsilon = 2^{-(j_1 + j'_1)}$ that the TPA,  $\tilde{\varphi}_{1_{J_1}}(d(x,y)^{-5})$ is computed to. We also perturb the $\alpha$ level of the Besov norm in each column and display results for different dimensions $M = 128, 256, 512$. The relative variation is less than 1, meaning the numerical verification of our analytical results of corollary \ref{cor:4.3} are invariant to the experimental parameters of $(\epsilon,\alpha,M)$, confirming increased local regularity gains of $\tilde{\varphi}_{1_{J_1}}(d(x,y)^{-5})$}.
\label{tab:1}
\end{table}

\subsection{Compression}

Recall from equation \ref{eq:25} that $K_n(x,y)$ has the representation $\tilde{\varphi}_{1_{(N,N')}}(d(x,y)^{-(n+1)}) + \Delta_{(N,N')}(\tilde{\varphi}_1, d(x,y)^{-(n+1)})$ following direct quasilinearization. We wish to compute $g(x) = \int \tilde{\varphi}_{1_{(N,N')}}(d(x,y)^{-5}) + \Delta_{(N,N')}(\tilde{\varphi}_1, d(x,y)^{-5}) f(y)$ rapidly by thresholding matrix entries of $\tilde{\varphi}_{1_{(N,N')}}(d(x,y)^{-5}) $ less than $\delta$ such that we end up with the matrix vector multiplication, $g(x) = \int \tilde{\varphi}_{1_{(N,N')}}(d(x,y)^{-5})_{\delta} + \Delta_{(N,N')}(\tilde{\varphi}_1, d(x,y)^{-5}) f(y)$. We generate 3 test functions, $f_1 = e^{-\frac{x - 0.5}{0.015}},x \in [0,1], f_2 = sin(x), x \in [0, 20 \pi] , f_3 = M(\frac{x}{M})^3 , x \in [0,1] $ (where $M$ is the tensor dimension for $f_3$) and compute the following matrix vector multiplications

\begin{align}
g_1(x) = (\int \tilde{\varphi}_{1_{\epsilon}}(d(x,y)^{-5})_{\delta} + \Delta_{\epsilon}(\tilde{\varphi}_1, d(x,y)^{-5})) f_1(y)
\label{eq:31}
\end{align}

\begin{align}
g_2(x) = \int (\tilde{\varphi}_{1_{\epsilon}}(d(x,y)^{-5})_{\delta} + \Delta_{\epsilon}(\tilde{\varphi}_1, d(x,y)^{-5})) f_2(y)
\label{eq:32}
\end{align}

\begin{align}
g_3(x) = \int (\tilde{\varphi}_{1_{\epsilon}}(d(x,y)^{-5})_{\delta} + \Delta_{\epsilon}(\tilde{\varphi}_1, d(x,y)^{-5})) f_3(y)
\label{eq:33}
\end{align}

for different combinations of precision, $\epsilon$, compression threshold, $\delta$ , and different dimensions $M = 128, 256, 512$ and visualize the test functions, $f_1, f_2, f_3$, and the outputs, $g_1, g_2, g_3$, in Figure \ref{fig:3} and report the relative error and compression rations in Table \ref{tab:2}

\begin{figure}[H]
    \centering
    \includegraphics[width=0.95\textwidth]{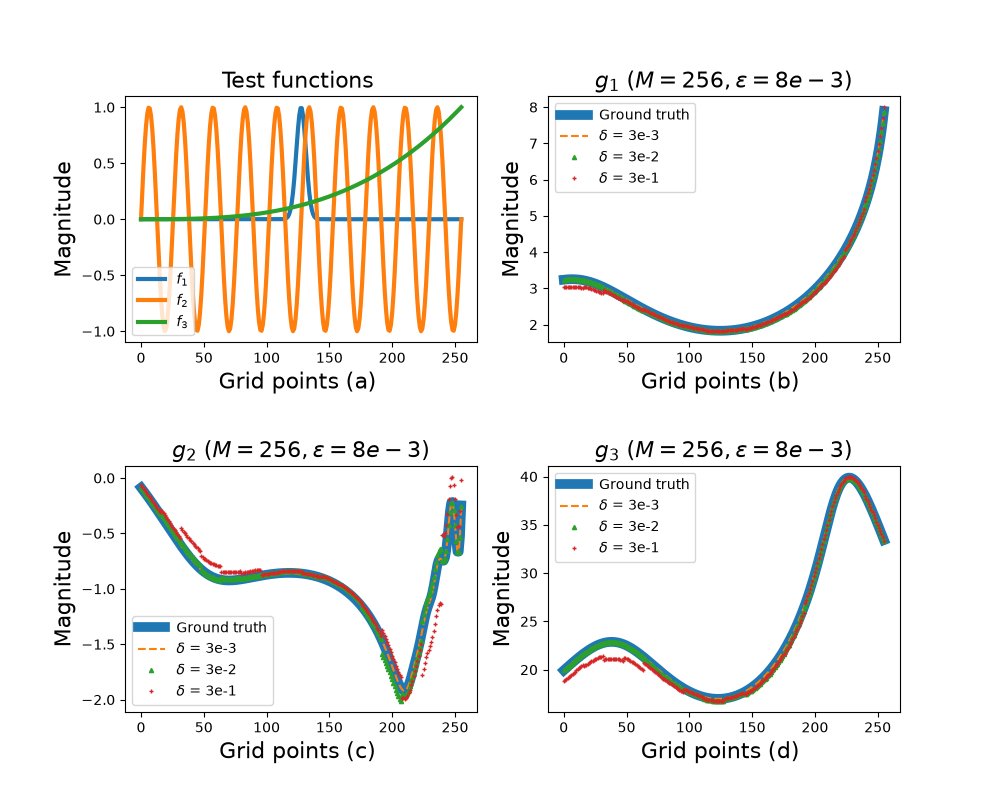}
    \caption{Potential kernel}
    \label{fig:3}
    \caption{(a) depicts the test functions, $f_1,f_2, f_3$, (b-d) depict $g_1,g_2,g_3$, respectively, which are the outputs of the matrix vector multiplications in the equations defined above, for different compression thresholds, $\delta = 3e-1,3e-2,3e-3$, following normalization with respect to the maximum of the absolute value of = $\tilde{\varphi}_{1_{\epsilon=8e-3}}(d(x,y)^{-5})$. The ground truth, $g(x) = \int K_5(x,y)f(y)$ is also included for $f_1,f_2,f_3$ in each plot. Qualitatively, one can see increasing the compression threshold, $\delta$, results in a less accurate compression.}
\end{figure}

\begin{table}[H]
\centering
\caption{Relative $L^2$, $L^{\infty}$ errors and compression rate (CR), for $ \delta = 3e-1$, and various precision levels, $\epsilon$, for test function $f_1$}.
\begin{tabular}{c|ccc}
\toprule
 & $ \lVert g_1 - \tilde{g_1} \rVert_{L^2} / \lVert g_1 \rVert_{L^2} $ & $ \lVert g_1 - \tilde{g_1} \rVert_{L^\infty} / \lVert g_1 \rVert_{L^\infty} $ & CR \\
\midrule
 & \multicolumn{3}{c}{M=128} \\
\midrule
$ \epsilon = 1.56\times10^{-2}$ & $5.31 \times 10^{-2}$ & $4.86 \times 10^{-2}$ & $6.40 \times 10^{1}$ \\
$ \epsilon = 7.81\times10^{-3} $ & $2.86 \times 10^{-2}$ & $2.73 \times 10^{-2}$ & $3.66 \times 10^{1}$ \\
$\epsilon = 3.91\times10^{-3} $ & $1.85 \times 10^{-2}$ & $1.50 \times 10^{-2}$ & $4.45 \times 10^{1}$ \\
\midrule
 & \multicolumn{3}{c}{M=256} \\
\midrule
$ \epsilon = 7.81\times10^{-3}$ & $2.78 \times 10^{-2}$ & $2.71 \times 10^{-2}$ & $4.27 \times 10^{1}$ \\
$ \epsilon = 3.91\times10^{-3}$ & $1.85 \times 10^{-2}$ & $1.49 \times 10^{-2}$ & $3.79 \times 10^{1}$ \\
$ \epsilon = 1.95\times10^{-3}$ & $2.67 \times 10^{-2}$ & $3.10 \times 10^{-2}$ & $ 9.31 \times 10^{1}$ \\
\midrule
 & \multicolumn{3}{c}{M=512} \\
\midrule
$ \epsilon = 3.91\times10^{-3} $  & $1.85 \times 10^{-2}$ & $1.48 \times 10^{-2}$ & $3.66 \times 10^{1}$ \\
$ \epsilon = 1.95\times10^{-3} $  & $2.53 \times 10^{-2}$ & $2.77 \times 10^{-2}$ & $9.75 \times 10^{1}$ \\
$ \epsilon = 9.77\times10^{-4} $ & $3.05 \times 10^{-2}$ & $3.24 \times 10^{-2}$ & $4.42 \times 10^{1}$ \\
\bottomrule
\end{tabular}
\caption{Included in the table are relative $L^2$ and $L^{\infty}$ errors for between the ground truth output, $g_1$, and the output associated with the compressed TPA, $\tilde{g}_1$, with $\delta = 0.3$. Also included is the compression ratio, CR, (ratio of the dense to sparse matrix). We perturb the precision $\epsilon$ as well as the dimension, $M$. }
\label{tab:2}
\end{table}

\section{Appendix}

Included here is an alternative direct quasilinearization of the potential kernel, $K_n(x,y)$, in the constant and variable coefficient cases. For this quaslinearization we consider seperating $K_n(x,y)$ into two non-linear transformations acting on $d(x,y)$ instead of one non-linear transformation acting on $\sum_{k=0}^{-(n+1)} a_k d(x,y)^k$. This perspective may be useful for situations where one wants to mitigate the effect of singularities of $d(x,y)$.

\subsubsection{Constant coefficient potential}

Consider $K_n(x,y)$ with $a_k = 1$ for $ 0 \leq k \leq n$. If we split $K_n(x,y)$ with $a_k = 1$ in terms of the operators $\varphi_{1,c}$ and $\varphi_2$, 

\begin{align}
\varphi_{1,c} : f \to \sum_{k=0}^n f^{-k} \qquad \varphi_2 : f \to \log(f)
\label{eq:34}
\end{align}

then the composition, $\varphi_3 \coloneq \varphi_2(\varphi_{1,c})$ permits us to understand $K_n(x,y)$ in terms of the operator $S_{(1,c),K_n}$:

\begin{align}
S_{(1,c),K_n}: d(x,y) \to \varphi_{3,c}(d(x,y))
\label{eq:35}
\end{align}

By proposition \ref{prop:5.1}, we can replace the operator, $S_{1,K_n}$, with the tensor paraproduct, $\Pi^{N,N'}_{\varphi'_{3,c} \varphi^{''}_{3,c}}$.

\begin{align}
\Pi^{N,N'}_{\varphi'_{3,c} \varphi^{''}_{3,c}} : d(x,y) \to \tilde{\varphi}_{3,c(N,N')}(d(x,y)) + \Delta_{(N,N')}(\varphi_{3,c}, d(x,y))
\label{eq:36}
\end{align}

and the approximation, $\tilde{\varphi}_{3,c(N,N')}(d(x,y))$, is obtained by computing the derivatives of $\tilde{\varphi}_{3,c(N,N')}(d(x,y))$. Collapsing the geometric series gives $\sum_{k=0}^n d(x,y)^{-k} \simeq d(x,y)^{-(n+1)} \epsilon^{-1}$ since $ d(x,y)^{-(n+1)} >> 1 $ such that 

\begin{align}
& \frac{d}{ d P^jP'^{j'}(d(x,y))} \varphi_{3,c(N,N')}(P^jP'^{j'}(d(x,y))) = \frac{-(n+1) \epsilon^{-1} P^jP'^{j'}(d(x,y))^{-(n+2)}}{\epsilon^{-1} P^jP'^{j'}(d(x,y))^{-(n+1)}}  \nonumber \\
&  = -(n + 1) P^jP'^{j'}(d(x,y))^{-1}
\label{eq:37}
\end{align}

\begin{align}
& \frac{ d^2 }{ d^2 P^jP'^{j'}(d(x,y))} \varphi_{3,c(N,N')}(P^jP'^{j'}(d(x,y))) = \frac{ d^2 }{ d^2 P^jP'^{j'}(d(x,y))} (-(n + 1) P^jP'^{j'}(d(x,y))^{-1} )  \nonumber \\
&  = (n + 1) P^jP'^{j'}(d(x,y))^{-2}
\label{eq:38}
\end{align}

giving the approximation

\begin{align}
\tilde{\varphi}_{3,c(N,N')}(d(x,y)) = \sum_{j,j'=0}^{N,N'} (n + 1)(P^jP'^{j'}(d(x,y))^{-2}Q^j Q'^{j'}(d(x,y)) - P^jP'^{j'}(d(x,y))^{-1}Q^jP'^{j'}(d(x,y))P^jQ'^{j'}(d(x,y)))
\label{eq:39}
\end{align}

\begin{remark}
Observe the terms $P^jP'^{j'}(d(x,y))^{-2}$ and $P^jP'^{j'}(d(x,y))^{-1}$in the preceeding formula lead to faster off-diagonal decay than the original kernel $K_n(x,y)$. 
\label{rem:9.1}
\end{remark}

A similar approximation is obtained when $a_k(x,y)$ in \ref{eq:1} is variable with mild regularity assumptions:

\subsubsection{Variable coefficient polynomial}

Consider $K_n(x,y)$ with $a_k(x,y) \in C^2(\mathbb{R})$, $a_k : [0,1] \to [-1,1]$ for $ 0 \leq k \leq n$. Again, we split the operators into $\varphi_{1,v}$ and $\varphi_2$,

\begin{align}
\varphi_{1,v} : f \to \sum_{k=0}^n f^{-k} \qquad \varphi_2 : f \to \log(f)
\label{eq:40}
\end{align}

then the composition, $\varphi_{3,v} \coloneq \varphi_2(\varphi_{1,v})$ permits us to understand $K_n(x,y)$ in terms of the operator $S_{(1,v),K_n}$:

\begin{align}
S_{(1,v),K_n}: d(x,y) \to \varphi_{3,v}(d(x,y))
\label{eq:41}
\end{align}

By proposition \ref{prop:5.1}, we can replace the operator, $S_{1,K_n}$, with the tensor paraproduct, $\Pi^{N,N'}_{\varphi'_{3,c} \varphi^{''}_{3,c}}$.

\begin{align}
\Pi^{N,N'}_{\varphi'_{3,v} \varphi^{''}_{3,v}} : d(x,y) \to \tilde{\varphi}_{3,v(N,N')}(d(x,y)) + \Delta_{(N,N')}(\varphi_{3,v}, d(x,y))
\label{eq:42}
\end{align}

and the approximation, $\tilde{A}_{(3,v)(N,N')}(d(x,y))$ is:

\begin{align}
& \tilde{\varphi}_{(3,v)(N,N')}(d(x,y)) = -(n+1)(P^jP'^{j'}(d(x,y)))^{-1}Q^jQ'^{j'}(d(x,y)) + \nonumber \\
& (((\epsilon - 1) (n^2 + 2n + 2) P^jP'^{j'}(d(x,y))^{-2}) + nP^jP'^{j'}(d(x,y))^{-2})Q^jP'^{j'}(d(x,y))P^jQ'^{j'}(d(x,y))
\label{eq:43}
\end{align}

Computing the derivatives $\varphi_{3,v}'$ and $\varphi_{3,v}''$ yields:

\begin{align}
& \frac{d \tilde{\varphi}_{(3,v)(N,N')} (P^jP'^{j'}(d(x,y)))}{ d P^jP'^{j'}(d(x,y)) }= \frac{\frac{d p_n( P^jP'^{j'}(d(x,y))) }{d P^jP'^{j'}(d(x,y)) }}{p_n(P^jP'^{j'}(d(x,y)))} \nonumber \\
& = \frac{-(n+1)a_{n+1} \epsilon^{-1}(P^jP'^{j'}(d(x,y)))^{-(n+2)}}{a_{n+1}(P^jP'^{j'}(d(x,y)))^{-(n+1)}\epsilon^{-1} - a_0\epsilon^{-1}}
\label{eq:44}
\end{align}

and 

\begin{align}
& \frac{ d^2 \tilde{\varphi}_{(3,v)(N,N')}(P^jP'^{j'}(d(x,y)))}{d^2 P^jP'^{j'}(d(x,y))} =  \frac{ p_n( P^jP'^{j'}(d(x,y))) \frac{d^2 p_n( P^jP'^{j'}(d(x,y))) }{d^2 P^jP'^{j'}(d(x,y))} - (\frac{d p_n( P^jP'^{j'}(d(x,y))) }{d P^jP'^{j'}(d(x,y)) })^2 }{(p_n(P^jP'^{j'}(d(x,y))))^2} \nonumber \\
 & = \frac{(n^2 + 3n + 2) \epsilon^{-1} a_{n+1}^2 P^jP'^{j'}(d(x,y)))^{-(2n+4)} }{((a_{n+1} P^jP'^{j'}(d(x,y))^{-(n+1)} - a_0) \epsilon^{-1})^2} - \frac{(n^2 + 3n + 2) \epsilon^{-1} a_{n+1} a_0 P^jP'^{j'}(d(x,y)))^{-(n+1)} }{((a_{n+1} P^jP'^{j'}(d(x,y))^{-(n+1)} - a_0) \epsilon^{-1})^2 } - \nonumber \\
 & \frac{(-(n+1) \epsilon^{-1}  a_{n+1} P^jP'^{j'}(d(x,y))^{-(n+2)})^2}{((a_{n+1} P^jP'^{j'}(d(x,y))^{-(n+1)} - a_0) \epsilon^{-1})^2 }
 \label{eq:45}
\end{align}

Under the assumptions on $a_k$ we can simplify $\varphi'_{(3,v)}$:

\begin{align}
& \frac{d \tilde{\varphi}_{(3,v)(N,N')} (P^jP'^{j'}(d(x,y)))}{ d P^jP'^{j'}(d(x,y)) } \simeq \frac{(-n+1)a_{n+1} \epsilon^{-1}(P^jP'^{j'}(d(x,y)))^{-(n+2)}}{a_{n+1} \epsilon^{-1}(P^jP'^{j'}(d(x,y)))^{-(n+1)}} \nonumber \\
& = -(n+1)(P^jP'^{j'}(d(x,y)))^{-1}
\label{eq:46}
\end{align}

which follows since $a_{n+1} P^jP'^{j'}(d(x,y))^{-(n+1)}\epsilon^{-1} >> a_0 \epsilon^{-1} $ since the range of $a_k$ is in $[-1,1]$. The same reasoning permits us to simplify $A''_3$: 

\begin{align}
& \frac{ d^2 \tilde{\varphi}_{(3,v)(N,N')}(P^jP'^{j'}(d(x,y)))}{d^2 P^jP'^{j'}(d(x,y))} \simeq  \frac{(n^2 + 3n + 2) \epsilon^{-1} a_{n+1}^2 P^jP'^{j'}(d(x,y)))^{-(2n+4)} }{ ((\epsilon^{-1}  a_{n+1} P^jP'^{j'}(d(x,y))^{-(n+1))^2})^2}  \nonumber \\
& - \frac{(n^2 + 3n + 2) \epsilon^{-1} a_{n+1} a_0 P^jP'^{j'}(d(x,y)))^{-(n+1)} }{  (\epsilon^{-1}  a_{n+1} P^jP'^{j'}(d(x,y))^{-(n+1)})^2  } - \frac{(-(n+1) a_{n+1} P^jP'^{j'}(d(x,y))^{-(n+2)})^2}{ (\epsilon^{-1}  a_{n+1} P^jP'^{j'}(d(x,y))^{-(n+1)})^2  } \nonumber \\
&= \frac{(n^2 + 3n + 2) \epsilon^{-1} a_{n+1}^2 P^jP'^{j'}(d(x,y)))^{-(2n+4)} }{ \epsilon^{-2}  a_{n+1}^2 P^jP'^{j'}(d(x,y))^{-(2n+2)}}  - \frac{(n^2 + 3n + 2)a_0}{\epsilon^{-1}  a_{n+1} P^jP'^{j'}(d(x,y))^{-(n+1)} }  - \nonumber \\
 & \frac{((n^2 + 2n + 2) \epsilon^{-2}  a_{n+1}^2 P^jP'^{j'}(d(x,y))^{-(2n+4)})}{\epsilon^{-2}  a_{n+1}^2 P^jP'^{j'}(d(x,y))^{-(2n+2)} } \nonumber \\
& = (n^2 + 3n + 2)\epsilon (P^jP'^{j'}(d(x,y))^{-2}  - a_0a_{n+1}^{-1}P^jP'^{j'}(d(x,y))^{(n+1)}) - (n^2 + 2n + 2)P^jP'^{j'}(d(x,y))^{-2} \nonumber \\
& = (\epsilon - 1) (n^2 + 2n + 2) (P^jP'^{j'}(d(x,y))^{-2} + n (P^jP'^{j'}(d(x,y))^{-2} - a_0a_{n+1}^{-1}P^jP'^{j'}(d(x,y))^{(n+1)}) \nonumber \\
& \simeq (\epsilon - 1) (n^2 + 2n + 2) (P^jP'^{j'}(d(x,y))^{-2} + n(P^jP'^{j'}(d(x,y))^{-2}
\label{eq:47}
\end{align}

where the last step holds since $P^jP'^{j'}(d(x,y))^{-2} >> P^jP'^{j'}(d(x,y))^{(n+1)}$. The last two steps permit us to compactly write the tensor paraproduct approximation:

\begin{align}
& \tilde{\varphi}_{(3,v)(N,N')}(d(x,y)) = \sum_{j,j'=0}^{N,N'} -(n+1)(P^jP'^{j'}(d(x,y)))^{-1}Q^jQ'^{j'}(d(x,y)) + \nonumber \\
& (((\epsilon - 1) (n^2 + 2n + 2) P^jP'^{j'}(d(x,y))^{-2}) + nP^jP'^{j'}(d(x,y))^{-2})Q^jP'^{j'}(d(x,y))P^jQ'^{j'}(d(x,y))
\label{eq:48}
\end{align}

\bibliographystyle{amsplain}
\bibliography{main}

@misc{stephane1999wavelet,
  title={A wavelet tour of signal processing},
  author={Stephane, Mallat},
  year={1999},
  publisher={Elsevier}
}

@article{fasina2025quasilinearization,
  title={Quasilinearization with regularizing tensor paraproducts},
  author={Fasina, Oluwadamilola},
  journal={arXiv preprint arXiv:2503.12629},
  year={2025}
}

@article{fasina2025d,
  title={D-tensor paraproducts and its caricatures},
  author={Fasina, Oluwadamilola},
  journal={arXiv preprint arXiv:2508.13322},
  year={2025}
}

@article{fasina2026hierarchical,
  title={Hierarchical paraproducts},
  author={Fasina, Oluwadamilola},
  journal={arXiv preprint arXiv:2602.16644},
  year={2026}
}

@inproceedings{bony1981calcul,
  title={Calcul symbolique et propagation des singularit{\'e}s pour les {\'e}quations aux d{\'e}riv{\'e}es partielles non lin{\'e}aires},
  author={Bony, Jean-Michel},
  booktitle={Annales scientifiques de l'{\'E}cole normale sup{\'e}rieure},
  volume={14},
  number={2},
  pages={209--246},
  year={1981}
}

@inproceedings{grengard2006rapid,
  title={The rapid evaluation of potential fields in three dimensions},
  author={Grengard, Leslie and Rokhlin, Vladimir},
  booktitle={Vortex Methods: Proceedings of the UCLA Workshop held in Los Angeles, May 20--22, 1987},
  pages={121--141},
  year={2006},
  organization={Springer}
}

@article{greengard2021fast,
  title={Fast multipole methods for the evaluation of layer potentials with locally-corrected quadratures},
  author={Greengard, Leslie and O'Neil, Michael and Rachh, Manas and Vico, Felipe},
  journal={Journal of Computational Physics: X},
  volume={10},
  pages={100092},
  year={2021},
  publisher={Elsevier}
}

@article{vico2016decoupled,
  title={The decoupled potential integral equation for time-harmonic electromagnetic scattering},
  author={Vico, Felipe and Ferrando, Miguel and Greengard, Leslie and Gimbutas, Zydrunas},
  journal={Communications on pure and applied mathematics},
  volume={69},
  number={4},
  pages={771--812},
  year={2016},
  publisher={Wiley Online Library}
}

@article{ankenman2018mixed,
  title={Mixed H{\"o}lder matrix discovery via wavelet shrinkage and Calder{\'o}n--Zygmund decompositions},
  author={Ankenman, Jerrod and Leeb, William},
  journal={Applied and Computational Harmonic Analysis},
  volume={45},
  number={3},
  pages={551--596},
  year={2018},
  publisher={Elsevier}
}

@article{beylkin1991fast,
  title={Fast wavelet transforms and numerical algorithms I},
  author={Beylkin, Gregory and Coifman, Ronald and Rokhlin, Vladimir},
  journal={Communications on pure and applied mathematics},
  volume={44},
  number={2},
  pages={141--183},
  year={1991},
  publisher={Wiley Online Library}
}

@article{barnett2011new,
  title={A new integral representation for quasi-periodic scattering problems in two dimensions},
  author={Barnett, Alex and Greengard, Leslie},
  journal={BIT Numerical mathematics},
  volume={51},
  number={1},
  pages={67--90},
  year={2011},
  publisher={Springer}
}

@article{gimbutas2003generalized,
  title={A generalized fast multipole method for nonoscillatory kernels},
  author={Gimbutas, Zydrunas and Rokhlin, Vladimir},
  journal={SIAM Journal on Scientific Computing},
  volume={24},
  number={3},
  pages={796--817},
  year={2003},
  publisher={SIAM}
}

@article{yarvin1998generalized,
  title={A generalized one-dimensional fast multipole method with application to filtering of spherical harmonics},
  author={Yarvin, Norman and Rokhlin, Vladimir},
  journal={Journal of Computational Physics},
  volume={147},
  number={2},
  pages={594--609},
  year={1998},
  publisher={Elsevier}
}

@book{hibschweiler2020fractional,
  title={Fractional Cauchy Transforms},
  author={Hibschweiler, Rita A and MacGregor, Thomas H},
  year={2020},
  publisher={CRC Press}
}

@article{cui2025enhanced,
  title={Enhanced batch adaptive filter based on fractional-order generalized cauchy kernel loss},
  author={Cui, Mingjing and Jiang, Yunxiang and Lin, Dongyuan and Wang, Shiyuan and He, Fuliang},
  journal={IEEE Signal Processing Letters},
  volume={32},
  pages={1201--1205},
  year={2025},
  publisher={IEEE}
}

@article{schlag2007remark,
  title={A remark on Littlewood-Paley theory for the distorted Fourier transform},
  author={Schlag, Wilhelm},
  journal={Proceedings of the American Mathematical Society},
  volume={135},
  number={2},
  pages={437--451},
  year={2007}
}

@article{goldstein2001holder,
  title={H{\"o}lder continuity of the integrated density of states for quasi-periodic Schr{\"o}dinger equations and averages of shifts of subharmonic functions},
  author={Goldstein, Michael and Schlag, Wilhelm},
  journal={Annals of Mathematics},
  pages={155--203},
  year={2001},
  publisher={JSTOR}
}

@article{erdougan2008strichartz,
  title={Strichartz and smoothing estimates for Schr{\"o}dinger operators with large magnetic potentials in $\mathbb{R}^3$},
  author={Erdo{\u{g}}an, M Burak and Goldberg, Michael and Schlag, Wilhelm},
  journal={Journal of the European Mathematical Society},
  volume={10},
  number={2},
  pages={507--531},
  year={2008}
}

@article{coifman1985some,
  title={Some new function spaces and their applications to harmonic analysis},
  author={Coifman, Ronald R and Meyer, Yves and Stein, Elias M},
  journal={Journal of functional Analysis},
  volume={62},
  number={2},
  pages={304--335},
  year={1985},
  publisher={Elsevier}
}

@article{coifman1986nonlinear,
  title={Nonlinear harmonic analysis, operator theory and PDE},
  author={Coifman, Ronald R and Meyer, Yves},
  journal={Beijing lectures in harmonic analysis},
  volume={112},
  pages={3--45},
  year={1986}
}

@incollection{coifman2011harmonic,
  title={Harmonic analysis of digital data bases},
  author={Coifman, Ronald R and Gavish, Matan},
  booktitle={Wavelets and Multiscale Analysis: Theory and Applications},
  pages={161--197},
  year={2011},
  publisher={Springer}
}

@article{ostermann2022fourier,
  title={Fourier integrator for periodic NLS: low regularity estimates via discrete Bourgain spaces},
  author={Ostermann, Alexander and Rousset, Fr{\'e}d{\'e}ric and Schratz, Katharina},
  journal={Journal of the European Mathematical Society},
  volume={25},
  number={10},
  pages={3913--3952},
  year={2022}
}

@article{marsden2026splitting,
  title={A splitting scheme for the wave maps equation at low regularity},
  author={Marsden, Katie and Rousset, Fr{\'e}d{\'e}ric and Schratz, Katharina},
  journal={arXiv preprint arXiv:2605.11507},
  year={2026}
}

@article{muscalu2004bi,
  title={Bi-parameter paraproducts},
  author={Muscalu, Camil and Pipher, Jill and Tao, Terence and Thiele, Christoph},
  year={2004}
}

@article{muscalu2006multi,
  title={Multi-parameter paraproducts},
  author={Muscalu, Camil and Pipher, Jill and Tao, Terence and Thiele, Christoph},
  year={2006}
}

@article{alazard2024paracomposition,
  title={Paracomposition operators and paradifferential reducibility},
  author={Alazard, Thomas and Shao, Chengyang},
  journal={arXiv preprint arXiv:2410.17211},
  year={2024}
}

@article{alazard2009paralinearization,
  title={Paralinearization of the Dirichlet to Neumann operator, and regularity of three-dimensional water waves},
  author={Alazard, Thomas and M{\'e}tivier, Guy},
  journal={Communications in Partial Differential Equations},
  volume={34},
  number={12},
  pages={1632--1704},
  year={2009},
  publisher={Taylor \& Francis}
}

@article{daubechies1997harmonic,
  title={Harmonic analysis, wavelets and applications},
  author={Daubechies, Ingrid C and Gilbert, CA},
  journal={New Jersey},
  year={1997}
}

@article{gilbert2003one,
  title={One-pass wavelet decompositions of data streams},
  author={Gilbert, Anna C. and Kotidis, Yannis and Muthukrishnan, S and Strauss, Martin J},
  journal={IEEE Transactions on knowledge and data engineering},
  volume={15},
  number={3},
  pages={541--554},
  year={2003},
  publisher={IEEE}
}

@article{beylkin1992representation,
  title={On the representation of operators in bases of compactly supported wavelets},
  author={Beylkin, Gregory},
  journal={SIAM Journal on Numerical Analysis},
  volume={29},
  number={6},
  pages={1716--1740},
  year={1992},
  publisher={SIAM}
}

@article{alpert1993wavelet,
  title={Wavelet-like bases for the fast solution of second-kind integral equations},
  author={Alpert, Bradley and Beylkin, Gregory and Coifman, Ronald and Rokhlin, Vladimir},
  journal={SIAM journal on Scientific Computing},
  volume={14},
  number={1},
  pages={159--184},
  year={1993},
  publisher={SIAM}
}

@article{hackbusch1999sparse,
  title={A sparse matrix arithmetic based on-matrices. Part I: Introduction to-matrices},
  author={Hackbusch, Wolfgang},
  journal={Computing},
  volume={62},
  number={2},
  pages={89--108},
  year={1999},
  publisher={Springer}
}

@article{kaye2018transparent,
  title={Transparent Boundary Conditions for the Time-Dependent Schr$\backslash$" odinger Equation with a Vector Potential},
  author={Kaye, Jason and Greengard, Leslie},
  journal={arXiv preprint arXiv:1812.04200},
  year={2018}
}

@article{wang2019monarch,
  title={Monarch butterfly optimization},
  author={Wang, Gai-Ge and Deb, Suash and Cui, Zhihua},
  journal={Neural computing and applications},
  volume={31},
  number={7},
  pages={1995--2014},
  year={2019},
  publisher={Springer}
}

@article{nelson2026bridging,
  title={Bridging the Simulation-to-Experiment Gap with Generative Models using Adversarial Distribution Alignment},
  author={Nelson, Kai and Kreiman, Tobias and Levine, Sergey and Krishnapriyan, Aditi S},
  journal={arXiv preprint arXiv:2604.01169},
  year={2026}
}

@article{psenka2026parallel,
  title={Parallel stochastic gradient-based planning for world models},
  author={Psenka, Michael and Rabbat, Michael and Krishnapriyan, Aditi and LeCun, Yann and Bar, Amir},
  journal={arXiv preprint arXiv:2602.00475},
  year={2026}
}

@article{demanet2007wave,
  title={Wave atoms and sparsity of oscillatory patterns},
  author={Demanet, Laurent and Ying, Lexing},
  journal={Applied and Computational Harmonic Analysis},
  volume={23},
  number={3},
  pages={368--387},
  year={2007},
  publisher={Elsevier}
}

@inproceedings{madhu2026heist,
  title={HEIST: a graph foundation model for spatial transcriptomics and proteomics data},
  author={Madhu, Hiren and Rocha, Jo{\~a}o Felipe and Huang, Tinglin and Viswanath, Siddharth and Krishnaswamy, Smita and Ying, Rex},
  booktitle={International Conference on Learning Representations},
  volume={2026},
  pages={65453--65476},
  year={2026}
}

@article{lindsey2026goforth,
  title={GoForth: Language Models for RNA Design under Structure, Sequence, and Coding Constraints},
  author={Lindsey, Michael},
  journal={arXiv preprint arXiv:2605.07608},
  year={2026}
}

@inproceedings{fesserunitary,
  title={Unitary Convolutions for Message-passing and Positional Encodings on Directed Graphs},
  author={Fesser, Lukas and Kiani, Bobak and Weber, Melanie},
  booktitle={Forty-third International Conference on Machine Learning}
}

@article{mateo2026lu,
  title={LU Factorization of Discrete Random Matrices},
  author={Mateo, Samuel Orellana and Urschel, John and West, Nicholas},
  journal={arXiv preprint arXiv:2608.08998},
  year={2026}
}

@article{wilber2025time,
  title={A time-frequency method for acoustic scattering with trapping},
  author={Wilber, Heather and Vaes, Wietse and Gopal, Abinand and Martinsson, Gunnar},
  journal={arXiv preprint arXiv:2506.15165},
  year={2025}
}

@article{trefethen2025numerical,
  title={Numerical conformal mapping},
  author={Trefethen, Lloyd N},
  journal={Notices Amer. Math. Soc},
  volume={72},
  pages={1300--1303},
  year={2025}
}

@article{dent2026controlling,
  title={How Controlling the Variance can Improve Training Stability of Sparsely Activated DNNs and CNNs},
  author={Dent, Emily and Tanner, Jared},
  journal={arXiv preprint arXiv:2602.05779},
  year={2026}
}

@article{cai2026globally,
  title={A Globally Convergent Third-Order Newton Method via Unified Semidefinite Programming Subproblems},
  author={Cai, Yubo and Zhu, Wenqi and Cartis, Coralia and Zardini, Gioele},
  journal={arXiv preprint arXiv:2603.09682},
  year={2026}
}

@article{kondor2025principles,
  title={The principles behind equivariant neural networks for physics and chemistry},
  author={Kondor, Risi},
  journal={Proceedings of the National Academy of Sciences},
  volume={122},
  number={41},
  pages={e2415656122},
  year={2025},
  publisher={National Academy of Sciences}
}

@article{gopal2025highly,
  title={A highly accurate procedure for computing globally optimal Wannier functions in one-dimensional crystalline insulators: A. Gopal, H. Zhang},
  author={Gopal, Abinand and Zhang, Hanwen},
  journal={Advances in Computational Mathematics},
  volume={51},
  number={6},
  pages={52},
  year={2025},
  publisher={Springer}
}

@article{zhang2024finding,
  title={Finding roots of complex analytic functions via generalized colleague matrices: H. Zhang and V. Rokhlin},
  author={Zhang, Hanwen and Rokhlin, Vladimir},
  journal={Advances in Computational Mathematics},
  volume={50},
  number={4},
  pages={71},
  year={2024},
  publisher={Springer}
}

\end{document}